\documentclass[a4paper,10pt]{amsart}
\usepackage[utf8]{inputenc}
\usepackage{amsxtra}
\usepackage{amsopn}
\usepackage{amsmath,amsthm,amssymb}
\usepackage{amscd}
\usepackage{array}
\usepackage{amsfonts}
\usepackage{latexsym}
\usepackage{verbatim}
\usepackage{marvosym}

\theoremstyle{plain}
\newtheorem{theorem}{Theorem}[section]
\newtheorem*{theorem*}{Theorem}

\newtheorem{prop}[theorem]{Proposition}

\newtheorem{rem}[theorem]{Remark}

\newtheorem{conj}[theorem]{Conjecture}
\newtheorem*{mt*}{Main Theorem}
\newcommand\R{{\mathbb R}}

\newcommand\Span{{\hbox{\em Span}}}
\newcommand{\Cpf}{$\mathcal{C}^\infty$-pure-and-full}

\title{Remarks on some compact symplectic solvmanifolds}
\author{Qiang Tan and Adriano Tomassini}

\address{Qiang Tan(\Letter): School of Mathematical Sciences, Jiangsu University, Zhenjiang, Jiangsu 212013, China}
\email{tanqiang@ujs.edu.cn }

\address{Adriano Tomassini: Dipartimento di Scienze Matematiche, Fisiche ed Informatiche,
Unit\`{a} di Matematica e Informatica\\
Universit\`{a} degli Studi di Parma\\
Parco Area delle Scienze 53/A, 43124\\
Parma, Italy}
\email{adriano.tomassini@unipr.it}

\thanks{This work is supported by PRC grant NSFC 11701226 (Tan);
  Natural Science Foundation of Jiangsu Province BK20170519 (Tan)
 and Project PRIN ``Varietà reali e complesse: geometria, topologia e analisi armonica''
and by GNSAGA of INdAM}

\begin{document}

\maketitle

\begin{abstract}
We study the hard Lefschetz property on compact symplectic solvmanifolds, i.e.,
compact quotients $M=\Gamma\backslash G$ of a
simply-connected solvable Lie group $G$ by a lattice
$\Gamma$, admitting a symplectic structure.
\end{abstract}

 \medskip
\section{Introduction}
Let $(M,\omega)$ be a compact symplectic $2n$-manifold, that is,
 $M$ is a $2n$-dimensional smooth manifold endowed with a closed non-degenerate $2$-form $\omega$,
 where $\omega$ is called the {\em symplectic structure}.
   We say that a symplectic manifold $(M,\omega)$ satisfies the {\em Hard Lefschetz Condition},
 shortly the {\em HLC}, if for any $k\in\{0,1,\cdot\cdot\cdot,n\}$, the homomorphism
$$
L^k:H^{n-k}_{dR}(M)\to H^{n+k}_{dR}(M)
$$
$$
[\alpha]\mapsto[\alpha\wedge\omega^k]
$$
is surjective (cf. \cite{Yan}). As a classical result, compact K\"ahler manifolds satisfy HLC; nevertheless, there are compact symplectic
manifolds satisfying HLC, with no K\"ahler structure. In analogy with Riemannian Geometry, starting with the symplectic form $\omega$,
one can define a symplectic codifferential $d^\Lambda$
operator; it turns out (see \cite{Merkulov,Mathieu,Yan}) that a compact symplectic manifold satisfies the HLC if and only if the $dd^\Lambda$-Lemma
holds, or, equivalently, any de Rham class $\mathfrak{a}$ of $M$ contains a symplectic harmonic representative $\alpha$,
i.e., $\mathfrak{a}=[\alpha]$ and $\alpha\in \ker d\cap \ker d^\Lambda$.

The aim of this paper is to study the Hard Lefschetz Condition on compact solvmanifolds, i.e.,
compact quotients $M=\Gamma\backslash G$ of a
simply-connected solvable Lie group $G$ by a lattice
$\Gamma$ endowed with a symplectic structure. We will show the following\smallskip

{\bf Theorem} (see Theorem \ref{4-dim})\  {\em Let $G$ be a simply-connected $4$-dimensional Lie group admitting a uniform lattice $\Gamma$ and let $M=\Gamma\backslash G$. Let
$\omega$ be a symplectic structure on $M$. Then, if $\omega$ satisfies the HLC,
any other symplectic structure on $M$ satisfies the HLC}.
\smallskip

Furthermore, for several examples of compact $6$-dimensional symplectic solvmanifolds, we show that the same conclusion of the above Theorem holds
(see Theorems \ref{HLC},
\ref{6-dim} and Remark \ref{rem-6-dim}).
\section{Preliminaries}
Let $(M,\omega)$ be a $2n$-dimensional symplectic manifold. An almost complex structure $J$ on $M$, i.e.,
a smooth $(1,1)$-tensor field on
$M$ satisfying $J^2=-\hbox{id}$, is said to be $\omega$-{\em compatible}, if at any given $x\in M$, for every pair of tangent
vectors $u,v$ and any non-zero tangent vector $w$, the following hold
$$
\omega_x(Ju,Jv)=\omega(u,v),\qquad \omega_x(w,Jw)>0.
$$
In other words, $g_x(u,v)=\omega_x(u,Jv)$ is an {\em almost K\"ahler metric} on $M$ and $(M,J,g,\omega)$
is an {\em almost K\"ahler manifold}.

Let $(M,J)$ be a $2n$-dimensional almost complex manifold and let $\Lambda^k(M)$ be the bundle of
$k$-forms on $M$; denote by $A^k(M)=\Gamma(M,\Lambda^k(M))$ the set of smooth sections of $\Lambda^k(M)$. Then
$J$ acts on $A^2(M)$ as
an involution, by setting $J\alpha(u,v)=\alpha(Ju,Jv)$, for every pair of vector fields $u$, $v$ on $M$.
Denote by $A^+_J(M)$, (resp. $A^-_J(M)$) the spaces of $J$-invariant, (resp.
$J$-anti-invariant) forms, i.e.,
$$
A_J^\pm(M)=\{\alpha\in A^2(M)\,\,\,\vert\,\,\,J\alpha=\pm \alpha\}.
$$
and by
$$
\mathcal{Z}^\pm_J=\{\alpha\in A^\pm_J(M)\,\,\,\vert\,\,\,d\alpha =0\}\,.
$$
Then, following T.-J. Li and W. Zhang \cite{li-zhang}, define
$$
H^\pm_J (M)= \left\{ \mathfrak{a} \in H^2_{dR}(M;\R) \,\,\, \vert \, \,\,\exists\,\alpha \in {\mathcal Z}^{\pm}_J\,\,\vert\,\,
\mathfrak{a}=[\alpha]\right\}.
$$
Then \cite[Definition 4.12]{li-zhang}, $J$ is said to be {\em\Cpf}\  if
$$
H^{2}_{dR} (M;\R) =H^{+}_J (M) \oplus H^{-}_J (X).
$$
If $(M,J,g,\omega)$ is a $2n$-dimensional almost K\"ahler manifold, then the space $A_0^2(M)$ of {\em primitive} $2$-{\em forms} is
given by
$$
A_0^2(M)=\{\alpha\in A^2(M)\,\,\,\vert\,\,\, \omega^{n-1}\wedge\alpha=0\}
$$ and the {\em primitive J-invariant cohomology group} $H^+_{J,0}(M)$ \cite{TWZ} by
$$
H^+_{J,0}(M)=\{\mathfrak{a} \in H^2_{dR}(M;\R) \,\,\, \vert \, \,\,\exists\,\alpha \in {\mathcal Z}^{+}_J\cap A_0^2(M) \,\,\vert\,\,
\mathfrak{a}=[\alpha]\}.
$$
We will denote by $P_J:A^2_0(M)\to A^2_0(M)$ the {\em generalized Lejmi differential
operator} (see \cite{L} and \cite{TWZ}) defined on the space $A^2_0(M)$
as
$$
P_J(\psi)=\Delta_g\psi-\frac{1}{n}g(\Delta_g\psi,\omega)\omega,
$$
where $\Delta_g$ is the Hodge Laplacian and $g(\cdot,\cdot\cdot)$ is the metric induced by $g$ on the space of $2$-forms.

In the sequel we will assume that $M$ is a compact symplectic solvmanifold, that is a compact quotient $M=\Gamma\backslash G$ of a
simply-connected solvable Lie group $G$ by a lattice $\Gamma$, endowed with a symplectic structure $\omega$. A compact solvmanifold $M=\Gamma\backslash G$
is said to be  {\em completely solvable} if the adjoint representation $\hbox{ad}_X$
of the Lie algebra $\mathfrak{g}$ of $G$ has real eigenvalues for every $X\in\mathfrak{g}$.
\section{The $dd^\Lambda$-Lemma on symplectic manifolds and symplectic cohomologies}
Let $(M,\omega)$ be a compact symplectic manifold of dimension $2n$.
Then the non degenerate $2$-form $\omega$
induces a $\mathcal{C}^\infty$-bilinear form $\omega^{-1}$ on $A^k(M)$, setting pointwise on simple elements
$$
\omega^{-1}(\alpha_1\wedge\cdots\wedge\alpha_k,\beta_1\wedge\cdots\wedge\beta_k)
=\det (\omega^{-1}(\alpha_i,\beta_j)),
$$
where $\omega^{-1}$ is the natural bilinear form induced by $\omega$ on $T^*M$ and then extending
$\omega^{-1}$ linearly on $A^k(M)$.
Then the {\em symplectic star operator} $*_s:A^k(M)\to A^{2n-k}(M)$
is defined by the following representation formula: given any $\beta\in A^k(M)$, for every $\alpha\in A^k(M)$, set
$$
\alpha\wedge *_s\beta=\omega^{-1}(\alpha,\beta)\frac{\omega^n}{n!}.
$$
It turns out that $*_s^2=\hbox{id}$. Denote as usual by $L$, $\Lambda$ and $H$ the three basic operators defined
respectively as
$$
\begin{array}{ll}
L:A^{k}(M)\to A^{k+2}(M), & L:=\omega\wedge\cdot\\[5pt]
\Lambda:A^{k}(M)\to A^{k-2}(M), & \Lambda:=*_s^{-1}L*_s\\[5pt]
H:A^{k}(M)\to A^{k}(M), & H:=[L,\Lambda].
\end{array}
$$
Then $\{L,\Lambda,H\}$ gives rise to an action of $\mathfrak{sl}_2$ on $A^*(M)=\bigoplus_{k\geq 0}A^k(M)$, namely a $\mathfrak{sl}_2$-{\em triple} on $A^*(M)$.

The {\em symplectic codifferential} $d^\Lambda$ is defined, acting on $k$-forms, by means of the following formula
\begin{equation}\label{dlambda}
 d^\Lambda\vert_{A^k(M)}=(-1)^{k+1}*_sd*_s.
\end{equation}
In view of the basic symplectic identity (see e.g., \cite{Brylinski}), it turns out that the symplectic
codifferential is expressed by
\begin{equation}\label{symplectic-identity}
d^\Lambda=[d,\Lambda].
\end{equation}
Since $*_s^2=\hbox{id}$, it is $(d^\Lambda)^2=0$; then Brylinski defined the following natural symplectic cohomology
$$
H^k_{d^\Lambda}(M):=\frac{\ker d^\Lambda\cap A^k(M)}{\hbox{\rm Im}\, d^\Lambda\cap A^k(M)},
$$
showing that the symplectic star operator $*_s$ induces an isomorphism between $H^k_{dR}(M)$ and
$H^{2n-k}_{d^\Lambda}(M)$. Later, Tseng and Yau \cite{TY} introduced the {\em Bott-Chern} and {\em Aeppli symplectic cohomologies} respctively as
$$
H^k_{d+d^\Lambda}(M):=\frac{\ker(d+d^\Lambda)\cap A^k(M)}{\hbox{\rm Im}\, dd^\Lambda\cap A^k(M)},
$$
and
$$
H^k_{dd^\Lambda}(M):=\frac{\ker(dd^\Lambda)\cap A^k(M)}{(\hbox{\rm Im}\,d+\hbox{\rm Im}\,d^\Lambda)\cap A^k(M)}.
$$
Then such cohomologies groups are the symplectic counterpart of the Bott-Chern and Aeppli cohomology groups
respectively defined in the complex setting. Tseng and Yau developed a Hodge theory for such cohomologies, showing that Bott-Chern and Aeppli symplectic cohomologies on a compact symplectic manifold are isomorphicto the kernel of suitable $4$-order elliptic self-adjoint differential operators. Consequently,
the symplectic cohomology groups are finite-dimensional vector spaces on a compact
symplectic manifold.

By definition, a compact symplectic manifold $(M,\omega)$ is said to satisfy the \emph{$dd^\Lambda$-lemma} (see \cite[Definition 3.12]{TY}) if the natural map
$H^\bullet_{d+d^\Lambda}(M)\to H^{\bullet}_{dR}(M)$ is injective, i.e.,
every $d^\Lambda$-closed, $d$-exact form
is also $dd^\Lambda$-exact, that is
\begin{equation}\label{ddLambda}
\ker d^\Lambda\cap \hbox{\rm Im}\,d= \hbox{\rm Im}\,dd^\Lambda.
\end{equation}
Since
$$
 *_s(\ker d^\Lambda\cap \hbox{\rm Im}\,d)=\ker d\cap \hbox{\rm Im}\,d^\Lambda,
$$
then \eqref{ddLambda} holds if and only if the following holds
\begin{equation}
 \ker d\cap \hbox{\rm Im}\,d^\Lambda= \hbox{\rm Im}\,dd^\Lambda.
\end{equation}
\\
We collect all the known results just recalling the following
\begin{theorem}
Let $(M,\omega)$ be a compact symplectic manifold. Then the following facts are equivalent:
\begin{itemize}
\item[i)] the Hard-Lefschetz condition holds on $(M,\omega)$, i.e., for every $k\in\mathbb{Z}$, the maps
$$
L^k:H^{n-k}_{dR}(M)\to H^{n+k}_{dR}(M)
$$
are isomorphisms;
\item[ii)] any de Rham cohomology class $\mathfrak{a}$ has a symplectic harmonic representative, i.e.,
$\mathfrak{a}=[\alpha]$, where
$d\alpha=0$ and $d^\Lambda\alpha=0$;
\item[iii)] the $dd^\Lambda$-lemma holds;
\item[iv)] the natural maps induced by the identity $H^{\bullet}_{d+d^\Lambda}(M)
\longrightarrow H^{\bullet}_{dR}(M)$ are injective;
\item[v)] the natural maps induced by the identity $H^{\bullet}_{d+d^\Lambda}(M)
\longrightarrow H^{\bullet}_{dd^\Lambda}(M)$ are isomorphisms.
%
\end{itemize}
\end{theorem}
For the proof see \cite[Conjecture 2.2.7]{Brylinski}, \cite[Corollary 2]{Mathieu}, \cite[Proposition 1.4]{Merkulov},
\cite[Theorem 0.1]{Yan},
\cite[Theorem 5.4]{C},
\cite[Proposition 3.13]{TY}
Therefore, according to the above Theorem, starting with a compact symplectic manifold $(M,\omega)$ satisfying the Hard Lefschetz Condition, then the de Rham cohomology algebra
$H^*_{dR}(M)=\bigoplus_{k\geq 0}H^k_{dR}(M)$ carries an $\mathfrak{sl}_2$-action. Indeed, by $ii)$, every de Rham cohomology
class of $M$
contains a symplectic harmonic
representative. Hence, the symplectic star operator $*_s$, and consequently, the operator
$\Lambda =*_s^{-1}L*_s$ are well defined on $H^*_{dR}(M)$,
by taking symplectic harmonic representatives.

Summing up, the de Rham cohomology of compact symplectic manifolds
satisfyng HLC shares with that of  K\"ahler manifolds an action of the Lie algebra $\mathfrak{sl}_2$. In the latter case and also in the Hyper-K\"ahler setting, Figueroa-O'Farrill, K\"ohl, and Spence \cite{FKS}, following an idea of Witten \cite{W}, showed that such an action and the Hodge-Lefschetz theory of compact (Hyper)-K\"ahler manifolds derive from the supersymmetry, more in particular, from the symmetries of certain supersymmetric sigma models (see also \cite{Z}).
Finally, it has to be remarked that HLC, or equivalently, the notion of $dd^\Lambda$-Lemma on compact symplectic manifolds, is a special case of the notion of the $dd^{\mathcal{J}}$-Lemma on generalized complex manifolds. For general results of such a notion and for other results in the context of supersymmetry we refer to \cite{C} and \cite{T} respectively.

\section{Hard Lefschetz Condition on 4-dimensional compact symplectic solvmanifolds}
Let $(M,\omega)$ be a compact symplectic manifold.
Then, it is easy to see that the set of all symplectic forms on $M$, $SF(M)$, is an open set in $\mathbb{R}^{b^2(M)}$,
where $b^2(M)$ is the second Betti number of $M$.
Let
$$
SF_{HLC}(M)\triangleq\{\Omega\in SF(M):(M,\Omega)\,\,\,{\rm has}\,\,\,{\rm HLC}\}.
$$
Since every K\"{a}hler manifold has the Hard Lefschetz Condition (cf. \cite{GH}),
then we have the following proposition
\begin{prop}
Every symplectic form on torus of dimension $2n$ is cohomologous to a K\"{a}hler form and thus has the Hard Lefschetz Condition.
\end{prop}
\begin{proof}
Let $\omega$ be a symplectic form on the torus, then it cohomologous to a sympplectic form $\omega_0$ which can be
expressed as a constant coefficient combination of the standard bases $\{ dx^i\wedge dx^j \}$.
For such an $\omega_0$ there exists a calibrated almost complex structure $J_0$, which is in fact integrable.
Therefore, $\omega_0$ is K\"{a}hler, and consequently it satisfies the HLC. The same holds for $\omega$.
\end{proof}

 K. Hasegawa has proven the following result in \cite{Hase}:
A compact solvmanifold admits  a K\"{a}hler structure if and only if it is a finite quotient of a complex torus which has a structure of
a complex torus bundle over a complex torus.
In particular, a compact solvmanifold of completely solvable type has a K\"{a}hler structure if and only if it is a complex torus.

\vskip 6pt

In the rest of this section we focus on Hard Lefschetz Condition $4$-dimensional compact homogeneous
manifolds $M=\Gamma\backslash G$, where $G$ is a simply-connected $4$-dimensional Lie group and $\Gamma$ is a
uniform lattice in $G$, endowed with a symplectic structure.
\begin{theorem}\label{4-dim}
Let $G$ be a simply-connected $4$-dimensional Lie group admitting a uniform lattice $\Gamma$ and let $M=\Gamma\backslash G$. Let
$\omega$ be a symplectic structure on $M$. Then, if $\omega$ satisfies the HLC,
any other symplectic structure on $M$ satisfies the HLC.
\end{theorem}
\begin{proof} Let $\mathfrak{g}$ be the Lie algebra of $G$. First of all,
recall that according to \cite[Theorem 9]{Chu}, a $4$-dimensional symplectic Lie algebra is solvable. Therefore,
$\mathfrak{g}$ is a unimodular and symplectic $4$-dimensional Lie algebra. According to \cite{Ov}, we have the following
list:
\begin{enumerate}
\item[0)] $\R^4$\vskip.2truecm\noindent
 \item[1)] $\mathfrak{nil}^3\times \R$;\vskip.2truecm\noindent
\item[2)] $\mathfrak{nil}^4$;\vskip.2truecm\noindent
\item[3)] $\mathfrak{sol}^3\times \R$;\vskip.2truecm\noindent
\item[4)] $\mathfrak{r}^{'}_{3,0}\times \R$.
\end{enumerate}
Denoting by $\{e^1,\ldots,e^4\}$ a basis of the dual space $\mathfrak{g}^*$, we can present the Lie algebras above by
the following Maurer-Cartan structure equations:
\begin{enumerate}
\item[0)] $\mathfrak{g}=\R^4$, $de^i=0$, $i=1,\ldots,4$.\\[10pt]
\item[1)] $\mathfrak{g}=\mathfrak{nil}^3\times \R$,
$$
de^1=0, \quad de^2=0, \quad de^3=-e^{12}, \quad de^4=0.
$$\\[10pt]
\item[2)] $\mathfrak{g}=\mathfrak{nil}^4$,
$$
de^1=e^{24}, \quad de^2=e^{34}, \quad de^3=0, \quad de^4=0.
$$\\[10pt]
\item[3)] $\mathfrak{g}=\mathfrak{sol}^3\times\R$,
$$
de^1=0, \quad de^2=e^{12}, \quad de^3=-e^{13}, \quad de^4=0.
$$\\[10pt]
\item[4)] $\mathfrak{g}=\mathfrak{r}'_{3,0}\times \R$,
$$
de^1=0, \quad de^2=-e^{13}, \quad de^3=e^{12}, \quad de^4=0,
$$
\end{enumerate}
where $e^{ij}:=e^i\wedge e^j$ and so on.
By assumption $\omega$ is a symplectic structure on $M$ satisfying the HLC. Therefore, in view of Benson and Gordon Theorem (cf. \cite{Benson,Benson2}), every
symplectic structure on any
compact quotient corresponding to case 1) and 2) does not satisfies HLC. The compact quotients corresponding to cases 0)
and 4) are complex solvmanifolds, namely {\em Complex Tori} and {\em Hyperelliptic Surfaces} respectively, so that any symplectic structure on such manifolds
satisfies the HLC. Finally, if $M=\Gamma\backslash G$ is a compact quotient corresponding to case 3),
we easily compute
\begin{enumerate}
\item[$\bullet$] $H^1_{dR}(M)=\Span_\R \langle [e^1],[e^4]\rangle $\\[10pt]
\item[$\bullet$] $H^2_{dR}(M)=\Span_\R \langle [e^{14}],[e^{23}]\rangle $;\\[10pt]
\item[$\bullet$] $H^3_{dR}(M)=\Span_\R \langle [e^{123}],[e^{234}]\rangle $.\\[10pt]
\end{enumerate}
Let $\omega_0=Ae^{14}+Be^{23}$, for $A,B\in\R$, $AB\neq 0$. Then $\omega_0$ is a symplectic structure on $M$ and it is immediate to check that it satisfies
the Hard Lefschetz condition. Let now $\omega$ be an arbitrary symplectic form on $M$. Then
$$
\omega = \omega_0+d\eta
$$
and, consequently, $\omega$ satisfies the HLC too.
\end{proof}

 \medskip

\section{Almost K\"ahler structures and Hard Lefschetz Condition on Nakamura manifolds}
The construction of completely solvable Nakamura manifolds (cf. \cite{Naka}) is well known.
For the sake of completeness we briefly recall it.
Let $A \in SL(2, \mathbb{Z})$ have two real positive distinct eigenvalues
\begin{displaymath}
\mu_1 =e^{\lambda}, \quad \mu_2 =e^{-\lambda}.
\end{displaymath}
Set
\begin{displaymath}
\Lambda = \left(\begin{array}{c c}
e^{-\lambda} & 0 \\
0 &  e^{\lambda}  \\
\end{array}\right)
\end{displaymath}
and let $P \in M_{2,2}(\R)$ be such that
\begin{displaymath}
\Lambda= P A P^{-1}
\end{displaymath}
Define $\Gamma := P\mathbb{Z}^2+ i P \mathbb{Z}^2$; then $\Gamma$ is a uniform discrete subgroup in $\mathbb{C}^2$, so that
$$
\mathbb{T}^2_{\mathbb{C}}= \mathbb{C}^2 / \Gamma
$$
is a $2$-dimensional complex torus and the map
\begin{align*}
& F : \mathbb{C}^2 \longrightarrow \mathbb{C}^2 \\
& F(z) = \Lambda z , \quad \textrm{where} \quad z=(z^1,z^2)^t,
\end{align*}
induces a biolomorphism of $\mathbb{T}^2_{\mathbb{C}}$ by setting $\tilde{F}([z])= [F(z)]$. Indeed, it is immediate to check that $\tilde{F}$ is well defined
and that $\tilde{F}$ is a biholomorphism, with $\tilde{F}^{-1}([z]) = [F^{-1}(z)]$. \\
By identifying $\R \times \mathbb{C}^2$ with $\R^5$ by $(s,z^1,z^2)\longmapsto (s, x^1,x^2,x^3,x^4)$, where
$z^1=x^1+i x^3$, $z^2= x^2+i x^4$, set
\begin{align*}
& T_1 : \R^5 \longrightarrow \R^5 \\
& T_1(s, x^1,x^2,x^3,x^4)= (s+\lambda, e^{-\lambda}x^1, e^{\lambda}x^2, e^{-\lambda}x^3, e^{\lambda}x^4),
\end{align*}
then $T_1(s,x^1,x^2,x^3,x^4)= T_1(s,z^1,z^2)= (s+\lambda, F(z^1,z^2))$. Therefore $T_1$ induces
a transformations of $\R \times \mathbb{T}^2_{\mathbb{C}}$, by setting
\begin{displaymath}
T_1(s,[(z^1,z^2)])=(s+\lambda , [F(z^1,z^2)]).
\end{displaymath}

Define
\begin{displaymath}
N^6 := \mathbb{S}^1\times \frac{\R \times
\mathbb{T}^2_{\mathbb{C}}}{<T_1>} . \vspace{4mm}
\end{displaymath}
Then $N^6$ is a compact $6$-dimensional solvmanifold of completely solvable type.

We give a numerical example.
Let
\begin{displaymath}
 A = \left(\begin{array}{c c}
3 & -1 \\
1 &  0  \\
\end{array}\right)
\end{displaymath}
 $A \in SL(2,\mathbb{Z})$.
Then $\mu_{1,2}= \frac{3\pm \sqrt{5}}{2}$. We set
\begin{displaymath}
\mu_1= \frac{3- \sqrt{5}}{2}= e^{-\lambda} \quad \textrm{and} \quad
\mu_2= \frac{3+\sqrt{5}}{2}= e^{\lambda},
\end{displaymath}
i.e., $\lambda= \log (\frac{3+\sqrt{5}}{2})$. Then
\begin{displaymath}
 P^{-1} = \left(\begin{array}{c c}
\frac{3-\sqrt{5}}{2} & \frac{3+\sqrt{5}}{2} \\
1 &  1  \\
\end{array}\right),
\end{displaymath}
and
\begin{displaymath}
 P = \left(\begin{array}{c c}
1 & -\frac{3+\sqrt{5}}{2} \\
-1 &  \frac{3-\sqrt{5}}{2}  \\
\end{array}\right)
\end{displaymath}
and the uniform lattice $\Gamma$ is given by
\begin{displaymath}
\Gamma= \Span_{\mathbb{Z}}<\left[\begin{array}{c}
-\frac{\sqrt{5}}{5} \\ \frac{\sqrt{5}}{5}\\ 0 \\ 0
\end{array}\right],
\left[\begin{array}{c}
\frac{5+3\sqrt{5}}{10}\\ \frac{5-3\sqrt{5}}{10} \\ 0 \\0
\end{array}\right],
\left[\begin{array}{c}
0\\ 0 \\ \frac{-\sqrt{5}}{5}\\ \frac{\sqrt{5}}{5}
\end{array}\right],
\left[\begin{array}{c}
0 \\  0  \\ \frac{5+3\sqrt{5}}{10}\\ \frac{5-3\sqrt{5}}{10}
\end{array}\right]>.
\end{displaymath}

By using previous notations, it is straightforward to check that
\begin{equation}\label{coframe-Nakamura}
\begin{cases}
e^1 := ds, \\
e^2 := dt, \\
e^3 := e^{s}dx^1,\\
e^4 := e^{-s}dx^2,\\
e^5 := e^{s} dx^3,\\
e^6 := e^{-s}dx^4.
\end{cases}
\end{equation}
gives rise to a global coframe on $N^6$,
where $dt$ is the global $1$-form on $\mathbb{S}^1$.
Therefore, with respect to $\{e^i\}_{i\in \{1, \ldots, 6\}}$ the structure equations are the following:
\begin{displaymath}
\begin{cases}
de^1=0,\\
de^2=0,\\
de^3=e^{13},\\
de^4=-e^{14},\\
de^5=e^{15}, \\
de^6=-e^{16}.
\end{cases}
\end{displaymath}
Then $(J,\omega,g)$ defined respectively as
\begin{equation}
\label{almost-complex-Nakamura}
\begin{cases}
J e^1 := -e^2,\\
Je^3 := -e^4,\\
Je^5 := -e^6,
\end{cases}
\end{equation}
\begin{equation}
\label{symplectic-Nakamura}
\omega := e^{12}+ e^{34}+ e^{56} ,
\end{equation}
and $g(\cdot,\cdot)=\omega(\cdot,J\cdot)$ give rise to an almost K\"ahler structure on $N^6$. Furthermore,
\begin{displaymath}
\begin{cases}
\psi ^1 := e^1+i e^2,\\
\psi ^2 := e^3+ie^4,\\
\psi ^3 := e^5+ie^6.
\end{cases}
\end{displaymath}
is a complex co-frame of $(1,0)$-forms on $(M^6,J)$; one can compute
\begin{displaymath}
\begin{cases}
d\psi ^1= 0,\\
d \psi ^2 = \frac{1}{2} (\psi ^{1\bar{2}}+ \psi ^{\bar{1}\bar{2}}),\\
d \psi ^3= \frac{1}{2}( \psi ^{1\bar{3}}+ \psi ^{\bar{1}\bar{3}}),
\end{cases}
\end{displaymath}
Since $b_1(N^6)=2$, $b_2(N^6)=5$ (see \cite{deT}), we obtain
\begin{align*}
H^1_{dR}(N^6)\simeq& \Span_\R <\psi ^{1}+\psi^{\bar{1}}, i(\psi ^{\bar{1}}-\psi^{1})>\\
H^2_{dR}(N^6)\simeq& \Span_\R <i \psi ^{1\bar{1}}, i \psi ^{2\bar{2}},i \psi ^{3\bar{3}},
i( \psi ^{2\bar{3}}+\psi ^{3\bar{2}})>, < i(\psi ^{23}-\psi^{\bar{2}\bar{3}})>,
\end{align*}
and consequently,
\begin{align*}
H^+_J(N^6) &= \R <i \psi ^{1\bar{1}}, i \psi ^{2\bar{2}},i \psi ^{3\bar{3}},
i( \psi ^{2\bar{3}}+\psi ^{3\bar{2}})> \\
H^-_J(N^6) &= \R < i(\psi ^{23}-\psi^{\bar{2}\bar{3}})>
\end{align*}
and the primitive $J$-invariant cohomology
$$
H^+_{J,0}(N^6)= <i(\psi ^{1\bar{1}}-\psi ^{3\bar{3}}), i(\psi ^{2\bar{2}}-\psi ^{3\bar{3}}),
i(\psi ^{2\bar{3}}+\psi ^{3 \bar{2}})>,
$$
so that the dimensions of such groups are
\begin{displaymath}
h^+_J=4, \quad h^-_J = 1, \quad h^+_{J,0}= 3.
\end{displaymath}
Then according to \cite[Proposition 2.3]{TWZ}
\begin{displaymath}
\dim_\R \ker P_J = h^+_{J,0}+h^-_J = 3 + 1 = 4 = b_2-1,
\end{displaymath}
and $J$ is $\mathcal{C}^{\infty}$-pure-and-full.
 We can show the following
\begin{theorem}\label{HLC}
Let $\Omega$ be any symplectic form on a Nakamura manifold $N^6$. Then $\Omega$ satisfies the HLC.
Thus, $SF_{HLC}(N^6)=SF(N^6)$.
\end{theorem}
\begin{proof}
The first and second de Rham cohomology groups of $N^6$ can be expressed in terms of the real coframe $\{e^1,\ldots,e^6\}$ as
\begin{align*}
H^1_{dR}(N^6)\simeq& \Span_\R  <e^1, e^2>\\
H^2_{dR}(N^6)\simeq& \Span_\R  <e^{12}, e^{34},e^{56},e^{36},e^{45}>.
\end{align*}
Let $\Omega$ be a symplectic form on $N^6$. Then $\Omega$ can be written as
$$
\Omega = c_1e^{12} +c_2e^{34} +c_3e^{56} +c_4 e^{36} +c_5e^{45} +d\eta\,,
$$
where $c_i\in\R,\,i=1,\ldots ,5$ and $\eta$ is a suitable $1$-form on $N^6$.
We obtain that
$$
\Omega^3 = 6c_1(c_2c_3+c_4c_5)e^{123456} +d\eta'.
$$
Since by assumption $\Omega$ is a symplectic structure on $N^6$, we get
\begin{equation}\label{symplectic-condition}
c_1(c_2c_3+c_4c_5)\neq 0.
\end{equation}
A direct computation shows that
$$
\begin{array}{ll}
{}&{}[\Omega]^2\cup [e^1]=[2(c_2c_3+c_4c_5)e^{13456})], \\[10pt]
{}&{}[\Omega]^2\cup [e^2]=[2(c_2c_3+c_4c_5)e^{23456})],
\end{array}
$$
and
$$
\begin{array}{ll}
{}&{}[\Omega]\cup [e^{12}]=[c_2e^{1234}+c_3e^{1256}+c_4e^{1236}+c_5e^{1245}], \\[10pt]
{}&{}[\Omega]\cup [e^{56}]=[c_1e^{1256}+c_2e^{3456}], \\[10pt]
{}&{}[\Omega]\cup [e^{34}]=[c_1e^{1234}+c_3e^{3456}], \\[10pt]
{}&{}[\Omega]\cup [e^{45}]=[c_1e^{1245}+c_4e^{3456}], \\[10pt]
{}&{}[\Omega]\cup [e^{36}]=[c_1e^{1236}+c_5e^{3456}].
\end{array}
$$
Therefore, by the above computations it turns out that
$$
\left\{
\begin{array}{ll}
[\Omega]^2& : H^1_{dR}(N^6)\to H^5_{dR}(N^6)\\[10pt]
[\Omega]& : H^2_{dR}(N^6)\to H^4_{dR}(N^6)
\end{array}
\right.
$$
are isomorphisms if and only if
$$
c_1(c_2c_3+c_4c_5)\neq 0,
$$
which is exactly the condition \eqref{symplectic-condition}. This ends the proof.
\end{proof}

 \medskip

\section{A six-dimensional cohomologically K\"{a}hler manifold with no K\"{a}hler metrics}

   Let $G(c)$ be the connected completely solvable Lie group of dimension $5$ consisting of matrices of the form
  \begin{equation*}
     a=\left(
     \begin{array}{cccccc}
       e^{cz} & 0 & 0 & 0 & 0 & x_1 \\
       0 & e^{cz} & 0 & 0 & 0 & y_1 \\
       0 & 0 & e^{cz} & 0 & 0 & x_2 \\
       0 & 0 & 0 & e^{cz} & 0 & y_2 \\
       0 & 0 & 0 & 0 & 1 & z \\
       0 & 0 & 0 & 0 & 0 & 1 \\
     \end{array}
    \right),
   \end{equation*}
   where $x_i,\,\,y_i,\,\,z\in\mathbb{R}$ $(i=1,2)$ and $c$ is a nonzero real number.
   Then a global system of coordinates $x_1,\,\,y_1,\,\,x_2,\,\,y_2$ and $z$ for $G(c)$ is given by
   $x_i(a)=x_i$, $y_i(a)=y_i$ and $z(a)=z$.
   A standard calculation shows that a basis for the right invariant $1$-forms on $G(c)$ consists of
   \begin{equation*}
     \{dx_1-cx_1dz,\,\,dy_1-cy_1dz,\,\,dx_2-cx_2dz,\,\,dy_2-cy_2dz,\,\,dz\}.
   \end{equation*}
   Alternatively, the Lie group $G(c)$ may be described as a semidirect product $G(c)=\mathbb{R}\ltimes_{\psi}\mathbb{R}^4$,
   where $\psi(z)$ is the linear transformation of $\mathbb{R}^4$ given by the matrix
   \begin{equation*}
    \left(
      \begin{array}{cccc}
        e^{cz} & 0 & 0 & 0 \\
        0 & e^{-cz}& 0 & 0 \\
        0 & 0 & e^{cz} & 0 \\
        0 & 0 & 0 & e^{-cz} \\
      \end{array}
    \right),
   \end{equation*}
for any $z\in\mathbb{R}$. Thus, $G(c)$ has a discrete subgroup
$$
\Gamma(c)=\mathbb{Z}\ltimes_{\psi}\mathbb{Z}^4
$$
such that the quotient space $\Gamma(c)\backslash G(c)$ is compact. Therefore, the forms
$$
dx_i-cx_idz, \quad dy_i-cy_idz,\qquad dz,\quad i=1,2
$$
descend to $1$-forms $\alpha_i$, $\beta_i$ and $\gamma$ $i=1,2$ on $\Gamma(c)\backslash G(c)$.

L.C. de Andr\'{e}s, M. Fern\'{a}ndez, M. de Le\'{o}n,
and J.J. Menc\'{\i}a considered the manifold $M^6(c)=G(c)/\Gamma(c)\times S^1$ which is a compact completely solvmanifold
(see \cite{dFdM}).
Moreover,  M. Fern\'{a}ndez, V. Mu\~{n}oz and J. A. Santisteban have proven that $M^6(c)$ does not admit any K\"{a}hler metric  {\rm(cf. \cite{FMS})}.
Here, there are $1$-forms $\alpha_1$, $\beta_1$, $\alpha_2$, $\beta_2$, $\gamma$ and $\eta$ on $M^6(c)$ such that
\begin{equation}
d\alpha_i=-c\alpha_i\wedge\gamma,\,\,\,d\beta_i=-c\beta_i\wedge\gamma,\,\,\,d\gamma=d\eta=0,
\end{equation}
where $i=1,2$ and such that at each point of $M^6(c)$, $\{\alpha_1, \beta_1, \alpha_2, \beta_2, \gamma,\eta\}$
is a basis for the $1$-forms on $M^6(c)$.
Using Hattori's theorem \cite{Hattori}, they compute the real cohomology of $M^6(c)$:
\begin{eqnarray*}
H^0_{dR}(M^6(c)) &\simeq& \Span_\R \langle 1\rangle, \nonumber\\
H^1_{dR}(M^6(c)) &\simeq& \Span_\R \langle [\gamma],\,\,[\eta]\rangle, \nonumber\\
H^2_{dR}(M^6(c)) &\simeq& \Span_\R \langle    [\alpha_1\wedge\beta_1],\,\,[\alpha_1\wedge\beta_2],\,\,
[\alpha_2\wedge\beta_1],\,\,[\alpha_2\wedge\beta_2],\,\,[\gamma\wedge\eta]\rangle, \nonumber\\
H^3_{dR}(M^6(c)) &\simeq& \Span_\R \langle[\alpha_1\wedge\beta_1\wedge\gamma],\,\,[\alpha_1\wedge\beta_2\wedge\gamma],\,\,[\alpha_2\wedge\beta_1\wedge\gamma],\,\,
[\alpha_2\wedge\beta_2\wedge\gamma], \nonumber\\
&& [\alpha_1\wedge\beta_1\wedge\eta],\,\,[\alpha_1\wedge\beta_2\wedge\eta],\,\,[\alpha_2\wedge\beta_1\wedge\eta],\,\,
[\alpha_2\wedge\beta_2\wedge\eta]\rangle, \nonumber\\
H^4_{dR}(M^6(c)) &\simeq& \Span_\R \langle [\alpha_1\wedge\beta_1\wedge\alpha_2\wedge\beta_2],\,\,[\alpha_1\wedge\beta_1\wedge\gamma\wedge\eta],\,\,
[\alpha_1\wedge\beta_2\wedge\gamma\wedge\eta], \nonumber\\
&& [\alpha_2\wedge\beta_1\wedge\gamma\wedge\eta],\,\,[\alpha_2\wedge\beta_2\wedge\gamma\wedge\eta] \rangle,\nonumber\\
H^5_{dR}(M^6(c)) &\simeq& \Span_\R \langle [\alpha_1\wedge\beta_1\wedge\alpha_2\wedge\beta_2\wedge\gamma],\,\,
[\alpha_1\wedge\beta_1\wedge\alpha_2\wedge\beta_2\wedge\eta]\rangle, \nonumber\\
H^6_{dR}(M^6(c)) &\simeq& \Span_\R \langle [\alpha_1\wedge\beta_1\wedge\alpha_2\wedge\beta_2\wedge\gamma\wedge\eta]\rangle.
\end{eqnarray*}
Therefore, the Betti number of $M^6(c)$ are
\begin{eqnarray*}
     b^0 &=& b^6=1,\nonumber\\
     b^1 &=& b^5=2,\nonumber \\
     b^2 &=& b^4=5, \nonumber\\
     b^3 &=& 8 .
\end{eqnarray*}

We denote by $(g,J,\omega)$ be an almost K\"{a}hler structure on $M^6(c)$, where we choose
$$
g=\alpha_1\otimes\alpha_1+\beta_1\otimes\beta_1+\alpha_2\otimes\alpha_2+\beta_2\otimes\beta_2+\gamma\otimes\gamma+\eta\otimes\eta
$$
and
$$
\omega=\alpha_1\wedge\beta_1+\alpha_2\wedge\beta_2+\gamma\wedge\eta.
$$
So $J$ is given by
$$
J\alpha_1=-\beta_1,\,\,\,J\alpha_2=-\beta_2,\,\,\,J\gamma=-\eta.
$$
It is clear that the maps
$$
[\omega]:H^2_{dR}(M^6(c))\rightarrow H^4_{dR}(M^6(c))
$$
and
$$
[\omega]^2:H^1_{dR}(M^6(c))\rightarrow H^5_{dR}(M^6(c))
$$
are isomorphisms. Thus, $(M^6(c),\omega)$ satisfies the Hard Lefschetz Condition.
By simple calculation, we can get
\begin{eqnarray*}
     H^-_J &=& Span_{\mathbb{R}}\{[\alpha_1\wedge\beta_2-\alpha_2\wedge\beta_1]\},  \\
    H^+_J &=& Span_{\mathbb{R}}\{[\alpha_1\wedge\beta_2+\alpha_2\wedge\beta_1],\,\,[\alpha_1\wedge\beta_1],\,\,
                  [\alpha_2\wedge\beta_2],\,\,[\gamma\wedge\eta]\},  \\
     \ker P_J&=& Span_{\mathbb{R}}\{\alpha_1\wedge\beta_2-\alpha_2\wedge\beta_1,\,\,\alpha_1\wedge\beta_2+\alpha_2\wedge\beta_1,  \nonumber \\
             &&  \alpha_1\wedge\beta_1-\gamma\wedge\eta,\,\,\alpha_2\wedge\beta_2-\gamma\wedge\eta   \}.
\end{eqnarray*}
Hence, $\dim\ker P_J=4=b^2-1$. Of course, $J$ is $C^\infty$ pure and full (cf. \cite{TWZ}).

In the following, we will study the other symplectic structures on $M^6(c)$.
   Set
   $$
   \xi_1=\alpha_1\wedge\beta_2+\alpha_2\wedge\beta_1,\,\,\,\xi_2=\alpha_1\wedge\beta_1-\gamma\wedge\eta
   ,\,\,\,\xi_3=\alpha_2\wedge\beta_2-\gamma\wedge\eta
   $$
  and $\theta=\alpha_1\wedge\beta_2-\alpha_2\wedge\beta_1$.
  Let $\Omega=c\omega+c_1\xi_1+c_2\xi_2+c_3\xi_3+a\theta$, where $c,c_i,a\in\mathbb{R}$.
  By direct calculation,
  \begin{eqnarray*}
    \Omega^3 &=& 6(c^3-cc_1^2-cc_2^2-cc_3^2+ca^2-cc_2c_3\\
     && +c_1^2c_2+c_1^2c_3-c_2c_3^2-c_2a^2-c_2^2c_3-c_3a^2)\alpha_1\wedge\beta_1\wedge\alpha_2\wedge\beta_2\wedge\gamma\wedge\eta.
  \end{eqnarray*}
  Then $\Omega$ is a symplectic form if and only if
  \begin{equation}\label{sym condition}
    c^3-cc_1^2-cc_2^2-cc_3^2+ca^2-cc_2c_3+c_1^2c_2+c_1^2c_3-c_2c_3^2-c_2a^2-c_2^2c_3-c_3a^2\neq  0.
  \end{equation}
  A direct computation shows that
$$
\begin{array}{ll}
{}&{}[\Omega]^2\cup [\gamma]=[2(c^2-c_1^2+a^2+cc_2+cc_3+c_2c_3)\alpha_1\wedge\beta_1\wedge\alpha_2\wedge\beta_2\wedge\gamma], \\[10pt]
{}&{}[\Omega]^2\cup [\eta]=[2(c^2-c_1^2+a^2+cc_2+cc_3+c_2c_3)\alpha_1\wedge\beta_1\wedge\alpha_2\wedge\beta_2\wedge\eta],
\end{array}
$$
and
$$
\begin{array}{ll}
{}&{}[\Omega]\cup [\alpha_1\wedge\beta_1]=[(c-c_2-c_3)\alpha_1\wedge\beta_1\wedge\gamma\wedge\eta+(c+c_3)\alpha_1\wedge\beta_1\wedge\alpha_2\wedge\beta_2], \\[10pt]
{}&{}[\Omega]\cup [\alpha_1\wedge\beta_2]=[(c-c_2-c_3)\alpha_1\wedge\beta_2\wedge\gamma\wedge\eta+(a-c_1)\alpha_1\wedge\beta_1\wedge\alpha_2\wedge\beta_2], \\[10pt]
{}&{}[\Omega]\cup [\alpha_2\wedge\beta_1]=[(c-c_2-c_3)\alpha_2\wedge\beta_1\wedge\gamma\wedge\eta-(a+c_1)\alpha_1\wedge\beta_1\wedge\alpha_2\wedge\beta_2], \\[10pt]
{}&{}[\Omega]\cup [\alpha_2\wedge\beta_2]=[(c-c_2-c_3)\alpha_2\wedge\beta_2\wedge\gamma\wedge\eta+(c+c_2)\alpha_1\wedge\beta_1\wedge\alpha_2\wedge\beta_2], \\[10pt]

{}&{}[\Omega]\cup [\gamma\wedge\eta]=[(c+c_2)\alpha_1\wedge\beta_1\wedge\gamma\wedge\eta+(c+c_3)\alpha_2\wedge\beta_2\wedge\gamma\wedge\eta
\\[10pt]
{}&{} \,\,\,\,\,\,\,\,\,\,\,\,\,\,\,\,\,\,\,\,\,\,\,\,\,\,\,\,\,\,\,\,\,\,\,\,\,\,\,\,\,\,\,\,\,\,\,\,
           +(c_1+a)\alpha_1\wedge\beta_2\wedge\gamma\wedge\eta+(c_1-a)\alpha_2\wedge\beta_1\wedge\gamma\wedge\eta ].
\end{array}
$$
 It turns out that
  $$
  \left\{
  \begin{array}{ll}
  [\Omega]^2& : H^1_{dR}(M^6(c))\to H^5_{dR}(M^6(c))\\[10pt]
  [\Omega]& : H^2_{dR}(M^6(c))\to H^4_{dR}(M^6(c))
   \end{array}
  \right.
  $$
   are isomorphisms if and only if
   \begin{equation}\label{HLC condition}
    (c-c_2-c_3)(c^2-c_1^2+a^2+cc_2+cc_3+c_2c_3)\neq 0,
   \end{equation}
   which is exactly the condition \eqref{sym condition}.
    It means that all symplectic forms on $M^6(c)$ satisfy the Hard Lefschetz Condition. Therefore, we proved the following
    \begin{theorem}\label{6-dim}
    Let $\Omega$ be any symplectic form on $M^6(c)$. Then $\Omega$ satisfies the HLC.
    Thus, $SF_{HLC}(M^6(c))=SF(M^6(c))$.
     \end{theorem}

     \begin{rem}\label{rem-6-dim}
     In addition, we still calculate two other compact completely solvmanifolds $N^6(c)$ and $P^6(c)$ in \cite{FMS}.
     We have found that all symplectic forms on these two manifolds satisfy the Hard Lefschetz Condition.
     \end{rem}

  At last, we want to propose the following conjecture:
       \begin{conj}
       Suppose that $M$ is a compact completely solvable manifold.
           Let $\omega$ be a symplectic structure on $M$. If $\omega$ satisfies the HLC,
                 then any other symplectic structure on $M$ satisfies the HLC.
                 \end{conj}

     \vskip 12pt

  \noindent{ {\bf Acknowledgments.} The first author would like to thank professor Hongyu Wang for stimulating
                              discussions.}

\end{document}